\documentclass[reqno]{amsart}

\usepackage{amsmath,amssymb,amsthm}
\usepackage{subfigure,cite,doi,hyperref}
\usepackage{comment}
\usepackage{graphicx,tikz}
\newtheorem{theorem}{Theorem}

\newtheorem{proposition}{Proposition}
\newtheorem{corollary}{Corollary}

\newtheorem{lemma}{Lemma}

\theoremstyle{definition}

\DeclareMathOperator{\diam}{diam}

\newcommand{\<}{\langle}
\renewcommand{\>}{\rangle}

\renewcommand{\a}{\alpha}

\renewcommand{\l}{\lambda}

\newcommand{\mynull}{\mathop{\mathrm{null}}\nolimits}
\newcommand{\myimage}{\mathop{\mathrm{Im}}\nolimits}
\newcommand{\rank}{\mathop{\mathrm{rank}}\nolimits}

\newcommand{\bR}{{\mathbb R}}

\begin{document}

\title[]{On Steinerberger Curvature and Graph Distance Matrices}


\author{Wei-Chia Chen \and Mao-Pei Tsui}
\address{National Center for Theoretical Sciences, Mathematics Division, Taipei, 106, Taiwan\footnote{Present address: Department of Mathematics, Iowa State University, Ames, IA, 50011, United States}}
\email{achia0329@ncts.ntu.edu.tw\footnote{Present email: wcchen@iastate.edu}}
\address{Department of Mathematics, National Taiwan University, Taipei, 106, Taiwan; \newline National Center for Theoretical Sciences, Mathematics Division, Taipei, 106, Taiwan}
\email{maopei@math.ntu.edu.tw}

\begin{abstract} 
 Steinerberger proposed a notion of curvature on graphs involving the graph distance matrix (J. Graph Theory, 2023).
 We show that nonnegative curvature is almost preserved under three graph operations. We characterize the distance matrix and its null space after adding an edge between two graphs. Let $D$ be the graph distance matrix and $\mathbf{1}$ be the all-one vector. We provide a way to construct graphs so that the linear system $Dx = \mathbf{1}$ does not have a solution. 

\end{abstract}

\maketitle

\vspace{-8pt}

\section{Introduction} 
\subsection{Introduction.} 
In recent years, the concept of curvature has been introduced to finite connected graphs \cite{linluyau,Devriendt_2022, DEVRIENDT202468}. In his paper \cite{curvatureongraphs}, Steinerberger proposed a notion of curvature defined on the vertices of graphs. He shows that this notion of curvature satisfies several properties that a curvature would satisfy in differential geometry. These include a Bonnet--Myers theorem, a Cheng's theorem, and a Lichnerowicz theorem. 
A necessary condition for the curvature to have these properties is that the linear system $$Dx = \mathbf{1}$$ has a solution, where $D$ is the graph distance matrix and $\mathbf{1}$ is the all-one constant vector. Steinerberger investigated the $9059$ graphs in the Mathematica database and reported there are only $5$ graphs that the linear system does not have a solution \cite{curvatureongraphs}. This phenomenon leads to the intriguing question: why does the linear system of equations $Dx = \mathbf{1}$ tend to have a solution for most graphs?

\smallskip
In his subsequent work \cite{firsteig}, Steinerberger gave a sufficient condition for $Dx = \mathbf{1}$ to have a solution, in terms of the Perron root and the Perron eigenvector $\eta$ of $D$. A problem of this condition is that it degenerates to whether $D$ is invertible or not. This condition still motivates the study of $\<\eta, \mathbf{1}\>.$ 

\smallskip
In this paper, we extend the results in \cite{curvatureongraphs,firsteig}. The main results are as follows.
\begin{enumerate}
    \item We simplify the proof of the invariance of total curvature (\cite[Proposition 3]{curvatureongraphs}) without using von Neumann's Minimax Theorem. 
    \item  We show that nonnegative curvature will be preserved except for one or two vertices after adding an edge between two graphs, merging two graphs at a vertex, or removing a bridge from a graph.
    \item We characterize the distance matrix and its null space after adding an edge between two graphs.
    \item We show that if two graphs have the property that $Dx = \mathbf{1}$ has no solution, then after merging them at a vertex, the new graph has the same property.
    \item We provide a lower bound to $\<\eta, \mathbf{1}\>$ involving the number of leaves when the graph is a tree.
\end{enumerate}

 \subsection{Definitions.}
Let $G=(V,E)$ be a finite, connected graph. The \textit{distance} between two vertices $u$ and $v$, denoted as $d(u,v)$, is the length of a shortest path from $u$ to $v.$ The \textit{curvature} of graph $G$ proposed by Steinerberger in \cite{curvatureongraphs} is a function $\mu:V\to \bR$ so that for every vertex $u \in V,$ we have
$$
\sum_{v \in V} d(u,v)\mu(v) = |V|.
$$
Equivalently, if the vertices are $V=\{v_i:1\leq i\leq n\}$, by considering the vector $(\mu(v_1),...,\mu(v_n))$, the curvature of a graph is a vector $w$ satisfying
$$
Dw = n \cdot \mathbf{1},
$$where $D_{ij} = d(v_i,v_j)$ is the distance matrix of the graph.

\section{Main Results}
\subsection{Total Curvature and Bonnet--Myers Sharpness}
The following property of the invariance of the total curvature was proved by Steinerberger as a consequence of von Neumann's Minimax Theorem. Inspired by his remarks, we simplify the proof by using linear algebra.
\begin{theorem}[\cite{curvatureongraphs}] Let $G$ be a connected graph. Suppose there are $w_1,w_2 \in \bR^n_{\geq 0}$ so that $Dw_1=Dw_2=n\cdot\mathbf{1}.$ Then $\|w_1\|_1=\|w_2\|_1.$
\end{theorem}
\begin{proof}
We have $$n \< w_1,\mathbf{1}\>=w_1^TDw_2 = n\<w_2,\mathbf{1}\>.$$
\end{proof}
    From the proof above, Theorem 1 can be relaxed without the assumption of nonnegative curvatures: Let $G$ be a connected graph. If $w_1$ and $w_2$ are curvatures of $G$, then $\<w_1,\mathbf{1}\> = \<w_2,\mathbf{1}\>.$

    \smallskip
    The following proposition states if a graph admits a nonnegative curvature, then the curvature at a vertex cannot be larger than half of the total curvature.

    \begin{proposition}
    \label{prop:curv_upper_half_total}
        Let $G$ be a connected graph with nonnegative curvature $w.$ Namely, $Dw = n \cdot\mathbf{1}$ with $n = |V(G)|$ and $w \in \mathbb{R}^n_{\geq 0}.$ Then for each vertex $u \in V(G)$ we have $w(u) \leq \frac{1}{2}\|w\|_1.$ 
    \end{proposition}
    \begin{proof}
        Suppose for the contradiction that $w(u_1) > \frac{\|w\|_1}{2}$ for some vertex $u_1.$ Let $u_2$ be a neighbor of $u_1.$ Note that
        \begin{align*}
        0 = n - n &= \sum_{v \in V(G)}d(u_1,v)w(v) - \sum_{v \in V(G)}d(u_2,v)w(v) \\
        &= \sum_{v \in V(G)}(d(u_1,v)-d(u_2,v))w(v) \\
        &= -w(u_1) + \sum_{v \neq u_1}(d(u_1,v) - d(u_2,v))w(v).
        \end{align*}
        By the triangle inequality, 
        $$
        |d(u_1,v)- d(u_2,v)| \leq 1.
        $$ Therefore, since $w$ is a nonnegative curvature, we have 
        $$
        w(u_1) \leq \sum_{v\neq u_1}w(v) = \|w\|_1 - w(u_1) < \|w\|_1 - \frac{1}{2}\|w\|_1 = \frac{1}{2}\|w\|_1,
        $$ a contradiction. Thus, for every vertex $u,$ we have $w(u) \leq \frac{1}{2}\|w\|_1.$
    \end{proof}


\smallskip
The discrete Bonnet--Myers theorem in \cite{curvatureongraphs} states that if $G$ has a nonnegative curvature $w$ so that $Dw = n \cdot\mathbf{1}$ with $K = \min_i w_i \geq 0,$ then 
$$
\diam G \leq \frac{2n}{\|w\|_{1}}.
$$ In addition, if $\diam G\cdot K = 2$, then $G$ has a constant curvature. We say the graph $G$ is \textit{discrete Bonnet--Myers sharp} if $\diam G\cdot K = 2$. Inspired by \cite{CUSHING2020107188}, we find that the discrete Bonnet--Myers sharpness will be preserved under the Cartesian product.
\begin{proposition}
    Let $G_1,G_2$ be connected graphs with curvatures bounded below by $K_1,K_2 \geq 0,$ respectively. Suppose $G_1,G_2$ are discrete Bonnet--Myers sharp. Then the Cartesian product graph $G_1 \square G_2$ is discrete Bonnet--Myers sharp. 
\end{proposition}
\begin{proof}
    The discrete Bonnet--Myers theorem above implies that $G_1$ and $G_2$ have constant curvature $K_1,K_2 > 0.$
    By \cite[Proposition 2]{curvatureongraphs}, $G_1 \square G_2$ has constant curvature $K>0$ and 
    $$
    K = (\frac{1}{K_{1}}+\frac{1}{K_{2}})^{-1} = \frac{K_{1}K_{2}}{K_{1} + K_{2}}.
    $$
    Since $\diam (G_1 \square G_2) = \diam G_1 + \diam G_2,$ we have
    $$
    \diam (G_1 \square G_2) \cdot  K = 2.
    $$
    \end{proof}

\subsection{Bridging, Merging, and Cutting Graphs} 
Nonnegative curvature will be preserved except for at most two vertices under three basic graph operations. Let $G_1$ and $G_2$ be two graphs whose distance matrices are $D_1$ and $D_2$, respectively. Assume that $G_i$ has $n_i$ vertices for $i=1,2.$ We can \textit{bridge} $G_1$ and $G_2$ to obtain a larger graph by adding an edge between them. We can \textit{merge} $G_1$ and $G_2$ at a vertex to obtain a graph $H$ by first bridging $G_1$ and $G_2$ via an edge, then performing an edge contraction on this edge. Given a connected graph with a bridge, we can \textit{cut} the graph into two components by removing the bridge.

\begin{theorem}[Bridging Graphs] Suppose $G_1$ and $G_2$ have nonnegative curvature, namely, $D_iw_i=n_i\cdot\mathbf{1}$ holds for some $w_i\in\bR^{n_i}_{\geq 0}$ for $i=1,2$. Then the graph $G$ obtained by adding an edge $e=\{u,v\}$ between $G_1$ and $G_2$ has a curvature nonpositive at $u$ and $v$ and nonnegative at every other vertex.
\end{theorem}
    \begin{figure}[h]
        \centering
        \subfigure{\includegraphics[width = 0.32\textwidth]{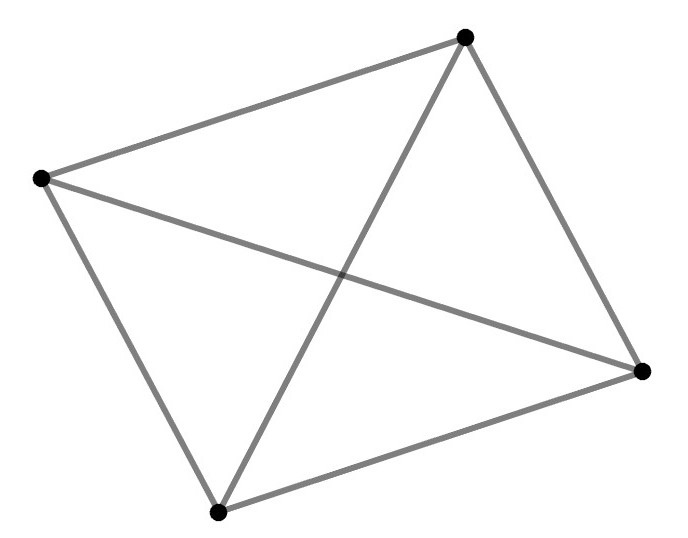}}
        \subfigure{\includegraphics[width = 0.32\textwidth]{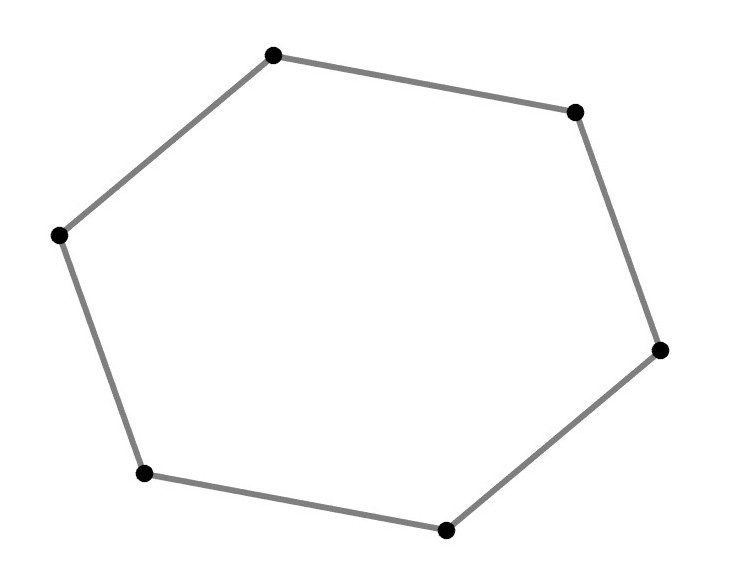}}
        \subfigure{\includegraphics[width = 0.32\textwidth]{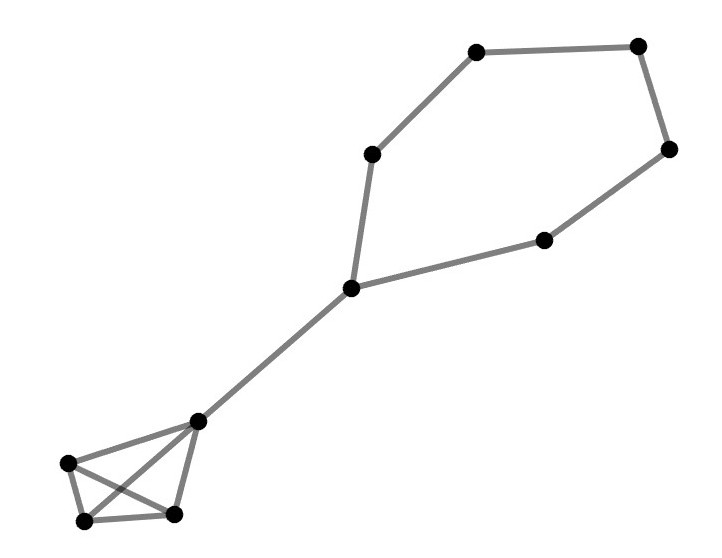}}
        \caption{Adding an edge between the complete graph $K_4$ and the cycle $C_6$.}
        \label{fig:C6 bridge K4}
    \end{figure}
As we will show, if $w_1,w_2$ are curvatures of $G_1$ and $G_2,$ respectively, then the curvature of $G$ at $u$ and at $v$ are
\begin{equation}
\label{eq:cuv_form_u_n1}
\frac{\|w_2\|_1(n_1+n_2)}{n_1\|w_2\|_1 + n_2\|w_1\|_1 + \frac{1}{2}\|w_1\|_1\|w_2\|_1}\cdot ((w_1)_{n_1}-\frac{1}{2}\|w_1\|_1)
\end{equation} and 
\begin{equation}
\label{eq:cuv_form_v_1}
\frac{\|w_1\|_1(n_1+n_2)}{n_1\|w_2\|_1 + n_2\|w_1\|_1 + \frac{1}{2}\|w_1\|_1\|w_2\|_1}\cdot ((w_2)_1 - \frac{1}{2}\|w_2\|_1),
\end{equation} respectively. By Proposition \ref{prop:curv_upper_half_total}, the curvatures at $u$ and at $v$ are nonpositive.

The curvature of $G$ at the endpoints of the bridge can indeed be negative. For example, consider adding an edge between two cycles $C_3$. The new graph has a unique curvature
$$
w = (\frac{12}{11},\frac{12}{11},\frac{-6}{11},\frac{-6}{11},\frac{12}{11},\frac{12}{11}).
$$

The proof of the above theorem was inspired by Bapat's work \cite{Bapat2010}. By using induction and decomposing the distance matrix into smaller ones, Bapat proved that for a tree, the distance matrix satisfies the equation
$$D\tau = (n-1)\cdot \mathbf{1}.$$ where $\tau = \mathbf{2} - (\deg(v_1),...\deg(v_n))^t$. 
We will relate the distance matrix of the larger graph $G$ to the distance matrices of $G_1$ and $G_2$ in our proof.

\begin{proof}[Proof of Theorem 2]
Let $V(G_1) = \{u_i : 1 \leq i \leq n_1\}$ and $V(G_2) = \{v_j:1 \leq j \leq n_2\}.$
The main observation is that if $u_i \in V(G_1)$ and $v_j \in V(G_2),$ then the shortest path from $u_i$ to $v_j$ has to pass through the edge $\{u,v\}.$ Relabel the vertices so that $u=u_{n_1}$ is the last vertex of $G_1$ and $v=v_1$ is the first vertex of $G_2.$ The observation implies
$$
d_G(u_i,v_j) = d_{G_1}(u_i,u_{n_1}) + 1 + d_{G_2}(v_1,v_j)
$$ for $1 \leq i \leq n_1, 1 \leq j \leq n_2.$
Let $y$ be the last column of $D_1$ and $z$ be the first column of $D_2.$ In other words,
\begin{align*}
    y_i &= d_{G_1}(u_i,u_{n_1}) \\
    z_j &= d_{G_2}(v_j,v_1)= d_{G_2}(v_1,v_j).
\end{align*}
If $D_G$ is the distance matrix of $G$, we can write 
$$
D_G = \begin{bmatrix}
    D_1 & y\mathbf{1}^t+\mathbf{1}\mathbf{1}^t+\mathbf{1}z^t \\
    \mathbf{1}y^t+ \mathbf{1}\mathbf{1}^t+z\mathbf{1}^t & D_2
\end{bmatrix}.
$$
Let $\a,s\in \bR$ be chosen later. Let $e_{n_1},e_{n_1+1} \in \bR^{n_1+n_2}$ be the $n_1$-th and the $(n_1+1)$-th standard coordinate vectors, respectively. 
Define
\begin{equation}
\label{eq:w}
w = \begin{bmatrix}
    \a w_1  \\
    w_2 
\end{bmatrix}+ se_{n_1}+se_{n_1+1}.
\end{equation}
Then $$
    D_Gw = \begin{bmatrix}
        \a n_1 \mathbf{1} + y\mathbf{1}^tw_2+ \mathbf{1}\mathbf{1}^tw_2 + \mathbf{1}z^tw_2 \\
        \a\mathbf{1}y^tw_1 + \a\mathbf{1}\mathbf{1}^tw_1 + \a z\mathbf{1}^tw_1 + n_2\mathbf{1}
    \end{bmatrix} + 
    \begin{bmatrix}
        s y \\
        s (z + \mathbf{1})
    \end{bmatrix} + 
    \begin{bmatrix}
        s (y + \mathbf{1}) \\
        s z
    \end{bmatrix}
    $$ since $z_1=y_{n_1}=0.$ By looking at the $n_1$-th row and the first row of $D_iw_i = n_i \cdot \mathbf{1}$, we have $y^tw_1=n_1$ and $z^tw_2=n_2.$ Therefore,
    $$
    D_Gw = \begin{bmatrix}
        (\a n_1 + \mathbf{1}^tw_2 + n_2 + s)\mathbf{1} + (2s + \mathbf{1}^tw_2)y \\
        (\a n_1 + n_2 + \a \mathbf{1}^tw_1 + s)\mathbf{1} + (2 s + \a\mathbf{1}^tw_1)z
    \end{bmatrix}.
    $$
    Define 
    $$
    s = \frac{-\mathbf{1}^tw_2}{2}, \a = \frac{\mathbf{1}^tw_2}{\mathbf{1}^tw_1} > 0.
    $$
    Note that since $\mathbf{1}^tw_1 > 0,$ the number $\a$ is well-defined. Then $2s = -\mathbf{1}^tw_2=-\a\mathbf{1}^tw_1.$
    Thus, we get
    $$
    D_Gw = (\a n_1+\mathbf{1}^tw_2 + n_2 + s)\mathbf{1} = (\a n_1 + n_2 + \frac{\mathbf{1}^tw_2}{2})\mathbf{1}.
    $$ This implies $G$ admits a curvature after scaling. We have
    $$
    \a n_1 + n_2 + \frac{\mathbf{1}^tw_2}{2} > 0.
    $$ Therefore, $G$ admits a curvature nonnegative everywhere except at the vertices $u_{n_1}$ and $v_1.$

    We now show equations \eqref{eq:cuv_form_u_n1} and \eqref{eq:cuv_form_v_1}. Note that by our construction and equation \eqref{eq:w}, the curvature at $u=u_{n_1}$ is
    \begin{align*}
   (\a\cdot (w_1)_{n_1}+s )&\cdot \frac{n_1+n_2}{\a n_1 + n_2 + \frac{\mathbf{1}^tw_2}{2}} \\
    &= (\frac{\|w_2\|_1}{\|w_1\|_1}\cdot (w_1)_{n_1}-\frac{\|w_2\|_1}{2}) \cdot \frac{n_1+n_2}{(\|w_2\|_1/\|w_1\|_1)\cdot n_1+n_2+\|w_2\|_1/2} \\
    &= \frac{\|w_2\|_1(n_1+n_2)}{n_1\|w_2\|_1 + n_2\|w_1\|_1+\frac{1}{2}\|w_1\|_1\|w_2\|_1}((w_1)_{n_1}-\frac{1}{2}\|w_1\|_1).
\end{align*} By similar computations, the curvature at $v = v_1$ is 
$$\frac{\|w_1\|_1(n_1+n_2)}{n_1\|w_2\|_1 + n_2\|w_1\|_1 + \frac{1}{2}\|w_1\|_1\|w_2\|_1}\cdot ((w_2)_1 - \frac{1}{2}\|w_2\|_1).$$ These show equations \eqref{eq:cuv_form_u_n1} and \eqref{eq:cuv_form_v_1}. By Proposition \ref{prop:curv_upper_half_total}, $$(w_1)_{n_1} - \frac{1}{2}\|w_1\|_1 \leq 0$$ and $$(w_2)_1 - \frac{1}{2}\|w_2\|_1 \leq 0.$$ Thus, the curvature of $G$ at $u$ and at $v$ are nonpositive.
    \end{proof}
    
\begin{corollary}
\label{cor:bridge-two-constant-curvature}
    Assume that $G'$ has constant curvature $K>0$ and $n = |V(G')|.$ Let $G$ be the graph obtained by adding an edge between two copies of $G'$. The curvature of $G$ has value $\frac{(2-n)2K}{4+K}<0$ at the vertices belonging to the edge and $\frac{4K}{4+K}>0$ at all the other vertices.
\end{corollary}
\begin{figure}[h!]
    \centering
    \includegraphics[width = 0.35\textwidth]{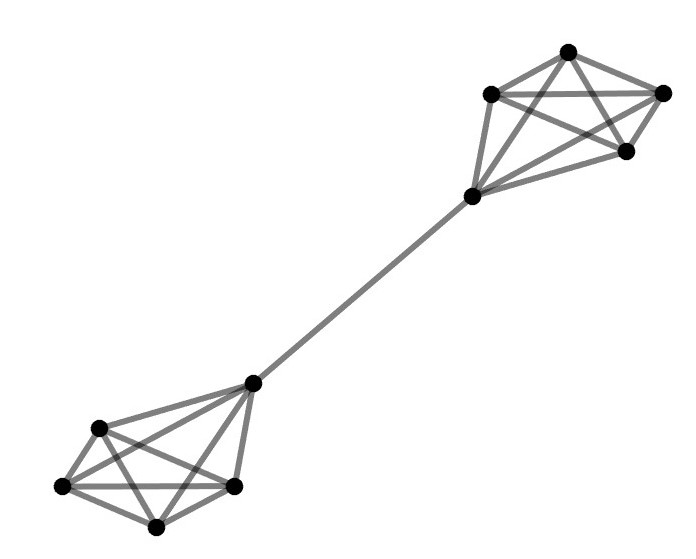}
    \caption{Adding an edge between two copies of $K_5$.}
    \label{fig:k5bridgek5}
\end{figure}

Corollary \ref{cor:bridge-two-constant-curvature} can be proved by plugging in $\a=1$ and $s = nK$ in equation \eqref{eq:w}.

The nonnegativeness of curvature will be preserved except for one vertex when we merge two graphs at this vertex.
\begin{theorem}[Merging Graphs] Suppose $G_1$ and $G_2$ have nonnegative curvature so that $D_iw_i=n_i\cdot\mathbf{1}$ for some $w_i \in \bR^{n_i}_{\geq 0}$. Then the graph $H$ obtained by adding an edge between $G_1$ and $G_2$ and performing an edge contraction on this edge has a curvature nonnegative everywhere except for the vertex of the contracted edge.
\end{theorem}
\begin{figure}[h]
    \centering
    \includegraphics[width= 0.35\textwidth]{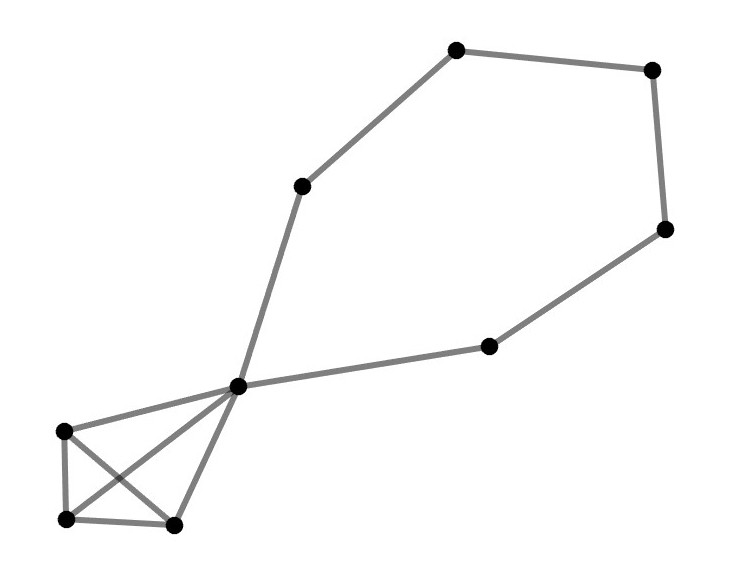}
    \caption{Merging the complete graph $K_4$ and the cycle $C_6$ at a vertex.}
    \label{fig:K4mergeC6}
\end{figure}

The idea is the same as the proof of Theorem 2. However, the analysis needs to be more careful.
\begin{proof}[Proof of Theorem 3]
    Write $V(G_1) = \{u_1,...,u_{n_1}\}$ and $V(G_2) = \{v_1,...,v_{n_2}\}$ so that the edge added and then contracted is $\{u_{n_1},v_1\}.$ Thus, $u_{n_1}$ and $v_1$ will be the identical vertex in $H.$ Let $y\in \bR^{n_1}$ be the last column of $D_1$ and $z\in \bR^{n_2-1}$ be the first column of $D_2$ without the first entry. Namely,
    $$
    y = \begin{bmatrix}
        d_{G_1}(u_1,u_{n_1}) \\
        \vdots \\
        d_{G_1}(u_{n_1},u_{n_1})
    \end{bmatrix}, 
    z = \begin{bmatrix}
        d_{G_2}(v_2,v_1) \\
        \vdots \\
        d_{G_2}(v_{n_2},v_1)
    \end{bmatrix}.
    $$ Let $w$ and $g$ be nonnegative vectors satisfying $D_1w = n_1\cdot \mathbf{1}$ and $D_2g = n_2 \cdot \mathbf{1}.$
    Let $\Bar{g} = (g_2,...,g_{n_2}),$ and $\Bar{D_2}$ be the matrix obtained by removing the first column and the first row of $D_2.$ Thus,
    $$
    D_2 = \begin{bmatrix}
        0 & z^t \\
        z & \Bar{D_2}
    \end{bmatrix}.
    $$
    The equation $D_2g = n_2 \cdot\mathbf{1}$ gives $z^t\Bar{g}=n_2$ and $\Bar{D_2}\Bar{g} = n_2\mathbf{1}-g_1z$. Similar to the proof of Theorem $2$, the shortest path in $H$ between $u_i$ and $v_j$ has to pass through the common vertex $u_{n_1}=v_1.$ We thus have
    $$
    d_H(u_i,v_j) = d_{G_1}(u_i,u_{n_1}) + d_{G_2}(v_1,v_j) = y_i + 1 + z_{j-1}
    $$
    for $1 \leq i \leq n_1$ and $2 \leq j \leq n_2$, since $d_{G_2}(v_1,v_j)=d_{G_2}(v_j,v_1).$ In addition,
    \begin{align*}
        d_H(u_i,u_j) &= d_{G_1}(u_i,u_j) \\
        d_H(v_i,v_j) &= d_{G_2}(v_i,v_j)
    \end{align*}
    hold for all $i,j.$ Therefore, we can write the distance matrix of $H$ as 
    $$
    D_H = \begin{bmatrix}
        D_1 & y\mathbf{1}^t + \mathbf{1}z^t \\
        \mathbf{1}y^t+z\mathbf{1}^t & \Bar{D_2}
    \end{bmatrix} \in \bR^{(n_1 + n_2 - 1)\times(n_1+n_2-1)}.
    $$
    Let $\a,s\in\bR$ be chosen later. Define a potential candidate of the curvature
    \begin{equation}
    \label{eq:w'}
    w' = \begin{bmatrix}
        \a w \\
        \mathbf{0}_{n_2-1}
    \end{bmatrix}+
    \begin{bmatrix}
        \mathbf{0}_{n_1} \\
        \Bar{g}
    \end{bmatrix} + 
    (s+g_1)e_{n_1} \in \bR^{n_1+n_2-1},
    \end{equation}
    where $e_{n_1}\in \bR^{n_1+n_2-1}$ is the $n_1$-th standard coordinate vector. Then
\begin{align*}
        D_Hw' &= \begin{bmatrix}
        \a n_1 \mathbf{1} \\
        \a \mathbf{1}y^tw + \a z \mathbf{1}^t w
    \end{bmatrix} + \begin{bmatrix}
        y\mathbf{1}^t\Bar{g} + \mathbf{1}z^t\Bar{g} \\
        n_2 \mathbf{1} - g_1 z
    \end{bmatrix} + (s+g_1)\begin{bmatrix}
        y \\ 
        z
    \end{bmatrix} \\
    &=
    \begin{bmatrix}
        (\a n_1 + z^t\Bar{g})\mathbf{1}\\
        (\a y^t w + n_2) \mathbf{1}
    \end{bmatrix} + \begin{bmatrix}
        (\mathbf{1}^t \Bar{g} + s + g_1)y \\
        (\a\mathbf{1}^t w + s)z
    \end{bmatrix}.
    \end{align*}
    Note that $z^t\Bar{g} = n_2$ and $y^tw = n_1.$ Set
    \begin{align*}
        s &= -g_1 - \mathbf{1}^t\Bar{g} = - \mathbf{1}^tg, \\
        \a &= \frac{-s}{\mathbf{1}^tw} = \frac{\mathbf{1}^tg}{\mathbf{1}^tw}.
    \end{align*}
    The fact that $w,g$ are nonnegative curvatures of $G_1$ and $G_2$, respectively, implies $\mathbf{1}^tw >0$ and $\mathbf{1}^tg>0$. Thus, $\a > 0$ is well-defined.
    We then have $$D_Hw' = (\a n_1 + n_2) \mathbf{1}.$$ Thus, we have
    $$
    D_H(\frac{n_1 + n_2 - 1}{\a n_1 + n_2}w') = (n_1+n_2-1)\cdot\mathbf{1}.
    $$This implies $H$ admits a curvature nonnegative everywhere except at the common vertex $u_{n_1}=v_1.$
\end{proof}
From our construction and equation \eqref{eq:w'}, the curvature of $H$ at the common vertex $u_{n_1}=v_1$ is 
$$
\frac{\|g\|_1(w)_{n_1}-\|w\|_1\|\Bar{g}\|_1}{\|g\|_1n_1+\|w\|_1n_2}\cdot (n_1+n_2-1).
$$

\smallskip
The following theorem states that nonnegative curvature will be preserved when we remove a bridge from a graph.
\begin{theorem}[Cutting Graphs] Suppose $G$ is a connected graph containing a bridge $e$. Let $G_1$ and $G_2$ be the components after removing $e$ from $G$. If $G$ has a nonnegative curvature then $G_1$ and $G_2$ have a nonnegative curvature. If $G$ has a constant curvature then $G_i$ has a constant curvature except at the vertices belonging to $e.$
\end{theorem}

\begin{proof}[Proof of Theorem 4]
Let $D_i$ be the distance matrices of $G_i$ for $i=1,2.$ Write $V(G_1) = \{u_1,...,u_{n_1}\}$ and $V(G_2) = \{v_1,...v_{n_2}\}$ so that the bridge is $e=\{u_{n_1},v_1\}.$ Since $G$ has a nonnegative curvature, we have 
$$
D_Gw = (n_1 + n_2) \cdot \mathbf{1}
$$
for some $w \in \bR^{n_1+n_2}_{\geq 0}.$ Write
$$
w = \begin{bmatrix}
    w_1 \\
    w_2
\end{bmatrix},
$$where $w_i \in \bR^{n_i}_{\geq 0}.$ Let $y$ be the last column of $D_1$ and $z$ be the first column of $D_2$, as in the proof of Theorem 2. Then
$$
D_G = \begin{bmatrix}
    D_1 & y \mathbf{1}^t + \mathbf{1}\mathbf{1}^t + \mathbf{1}z^t \\
    \mathbf{1}y^t + \mathbf{1}\mathbf{1}^t + z\mathbf{1}^t & D_2
\end{bmatrix}.
$$
Since $D_G w = (n_1+n_2)\cdot\mathbf{1},$ we get
\begin{align}
\label{eq:3}
    D_1w_1 + y\mathbf{1}^tw_2+\mathbf{1}\mathbf{1}^tw_2+ \mathbf{1}z^tw_2 &= (n_1+n_2)\cdot\mathbf{1} \\
    \label{eq:5}
    \mathbf{1}y^tw_1 + \mathbf{1}\mathbf{1}^tw_1 + z\mathbf{1}^tw_1 + D_2w_2 &= (n_1+n_2)\cdot\mathbf{1}.
\end{align}
The last row of equation \eqref{eq:3}, the first row of equation \eqref{eq:5}, together with $y_{n_1}=z_1=0$ give
\begin{align}
\label{eq:6}
    y^tw_1 + (z+\mathbf{1})^tw_2 &= n_1 + n_2 \\
    \label{eq:4}
    (\mathbf{1}+y)^tw_1 + z^tw_2 &= n_1 + n_2.
\end{align}
Define 
\begin{align*}
    \Bar{w_1}&=w_1 + (\mathbf{1}^t w_2) e_{n_1} \\
    \Bar{w_2} &= w_2 + (\mathbf{1}^t w_1) e_1,
\end{align*} where $e_{n_1},e_1$ are the $n_1$-th and the first coordinate vectors in $\bR^{n_1}$ and $\bR^{n_2}$, respectively.
Then 
\begin{align*}
D_1\Bar{w_1} &= D_1w_1 + \mathbf{1}^tw_2 y = (n_1+n_2-\mathbf{1}^tw_2-z^tw_2)\mathbf{1} = (y^tw_1)\mathbf{1} \\
D_2 \Bar{w_2} &= D_2w_2 + \mathbf{1}^tw_1 z = (n_1+n_2 - y^tw_1 - \mathbf{1}^tw_1)\mathbf{1} = (z^tw_2)\mathbf{1},
\end{align*}by equations \eqref{eq:3} to \eqref{eq:4}.

We claim that $y^tw_1,z^tw_2 >0.$ 
Suppose $y^tw_1 = 0.$ Since $w=\begin{bmatrix}
    w_1 \\
    w_2
\end{bmatrix}$ satisfies
$$
D_Gw = (n_1+n_2) \cdot \mathbf{1} 
$$ and $D_G$ is a nonnegative matrix, we have 
$$
\mathbf{1}^t w = \mathbf{1}^t w_1 + \mathbf{1}^t w_2 > 0.
$$
Note that equations \eqref{eq:6} and \eqref{eq:4} imply $\mathbf{1}^tw_1 = \mathbf{1}^t w_2.$ Therefore, $\mathbf{1}^tw_1 = \mathbf{1}^tw_2 > 0.$ Since $w_1 \in \bR^{n_1}_{\geq 0}$ and $y_{n_1} = 0,$ we have
$$
0 < \mathbf{1}^tw_1 \leq y^tw_1 + (w_1)_{n_1} = (w_1)_{n_1}.
$$
This implies $w_1 = ce_{n_1}$ with $c=(w_1)_{n_1}>0.$ Plugging this into equation \eqref{eq:3}, we get
\begin{align*}
    cy + (\mathbf{1}^tw_2)y= (n_1 + n_2 - \mathbf{1}^tw_2-z^tw_2)\cdot\mathbf{1} = (y^tw_1)\cdot \mathbf{1}=\mathbf{0},
\end{align*}
where the second equality follows by equation $\eqref{eq:6}.$ Since $\mathbf{1}^tw_2= \mathbf{1}^tw_1 = c,$ we get $2cy = \mathbf{0}.$
This implies $c = 0$ and $w_1=\mathbf{0}.$ Then $\mathbf{1}^tw_1=0,$ a contradiction. Therefore, we have $y^t w_1 > 0$. A similar argument shows that $z^tw_2 > 0.$

\smallskip
Consider
\begin{align*}
w_1' &= \frac{n_1}{y^tw_1}\Bar{w_1} \\
w_2' &= \frac{n_2}{z^tw_2}\Bar{w_2}.
\end{align*}
Then $D_iw_i'=n_i\cdot \mathbf{1}$ for $i=1,2.$ Thus, $G_i$ has a nonnegative curvature for $i=1,2.$

If both $w_1$ and $w_2$ are constant, then by construction, $w_1'$ and $w_2'$ are constant everywhere except at vertices $u_{n_1}$ and $v_1,$ respectively. 
\end{proof}

\subsection{Null Space of Graph Distance Matrix} 
In the previous section, we created a new graph $G$ by adding an edge between two graphs $G_1$ and $G_2$. In this section, we give a characterization of the null space of the distance matrix of $G.$
\begin{theorem}
    Let $G_1, G_2$ be two connected graphs admitting curvatures. Suppose $G_i$ has $n_i$ vertices for $i=1,2$. Let $G$ be the graph obtained by adding an edge between $G_1$ and $G_2$. Let $D_G,D_1,D_2$ be the distance matrices of $G,G_1,$ and $G_2$, respectively. Then we have
    $$
    \mynull D_G=\mynull D_1 \oplus \mynull D_2
    $$
    and 
    $$\dim \mynull D_G = \dim \mynull D_1 + \dim \mynull D_2,$$ where we canonically embed $\mynull D_i$ to $\bR^{n_1+n_2}$ by augmenting zeros.
\end{theorem}
This implies that $$
\rank D_G = \rank D_1 + \rank D_2.$$

\begin{proof}[Proof of Theorem 5]
    As in the proof of Theorem 2, we write
    $$
    D_G = \begin{bmatrix}
    D_1 & y\mathbf{1}^t+\mathbf{1}\mathbf{1}^t+\mathbf{1}z^t \\
    \mathbf{1}y^t+ \mathbf{1}\mathbf{1}^t+z\mathbf{1}^t & D_2
    \end{bmatrix},
$$where $y$ is the last column of $D_1,$ and $z$ is the first column of $D_2.$ Since $G_1$ and $G_2$ are nonnegatively curved, $\mathbf{1} \in \myimage D_i=(\mynull D_i) ^\perp.$ In addition, by Theorem $2$, $G$ admits a curvature. This implies $$\mathbf{1} \in \myimage D_G = (\mynull D_G)^\perp.$$ If $\eta \in \mynull D_1$, then $y^t\eta=\mathbf{1}^t\eta=0$. This implies $D_G\begin{bmatrix}
    \eta \\
    \mathbf{0}_{n_2}
\end{bmatrix} = \mathbf{0}.$ Similarly, if $\xi \in \mynull D_2$, then $\mathbf{1}^t\xi=z^t\xi=0$. This implies $D_G \begin{bmatrix}
    \mathbf{0}_{n_1} \\
    \xi
\end{bmatrix}= \mathbf{0}.$ Therefore, if $\{\eta_1,...,\eta_{k_1}\}$ is a basis of $\mynull D_1$ and $\{\xi_1,...,\xi_{k_2}\}$ is a basis of $\mynull D_2,$ then
$$\left\{\begin{bmatrix}
        \eta_1 \\
        \mathbf{0}_{n_2}
    \end{bmatrix},...,\begin{bmatrix}
        \eta_{k_1} \\
        \mathbf{0}_{n_2}
    \end{bmatrix}, \begin{bmatrix}
        \mathbf{0}_{n_1} \\
        \xi_1
    \end{bmatrix},...,\begin{bmatrix}
        \mathbf{0}_{n_1} \\
        \xi_{k_2}
    \end{bmatrix}\right\}$$ is linearly independent in $\mynull D_G .$ This shows that $$\dim \mynull D_G \geq k_1+k_2 = \dim \mynull D_1 + \dim \mynull D_2.$$

    On the other hand, suppose $\begin{bmatrix}
        \eta \\
        \xi
    \end{bmatrix} \in \mynull D_G.$ Our goal is to show that $\eta \in \mynull D_1$ and $\xi \in \mynull D_2$. We have
    \begin{align}
        \label{eq:1}
        \mathbf{0}_{n_1} &= D_1 \eta + y\mathbf{1}^t\xi + \mathbf{1}\mathbf{1}^t\xi  +\mathbf{1}z^t\xi\\
        \label{eq:2}
        \mathbf{0}_{n_2} &= \mathbf{1}y^t\eta  + \mathbf{1}\mathbf{1}^t\eta + z\mathbf{1}^t\eta + D_2\xi \\
        0 &= \mathbf{1}^t\begin{bmatrix}
            \eta \\
            \xi
        \end{bmatrix}. 
    \end{align}
By looking at the $n_1$-th row of the first equation and using $y_{n_1}=0,$ we get
$$
0 = y^t \eta  + \mathbf{1}^t \xi + z^t \xi.
$$
The first row of the second equation and $z_1=0$ gives
$$
0 = y^t\eta + \mathbf{1}^t \eta + z^t \xi.
$$
Combining these with the third equation, we conclude that 
$$
\mathbf{1}^t \eta = \mathbf{1}^t \xi = 0.
$$
Therefore, equations \eqref{eq:1} and \eqref{eq:2} give
\begin{align*}
D_1 \eta &= -(z^t \xi)  \mathbf{1} \\
D_2 \xi &= -(y^t\eta)  \mathbf{1}.
\end{align*}
Suppose that $z^t\xi \neq 0.$ Since $G_1$ admits a nonnegative curvature $w_1$ and at least one of its entries is not zero, by the invariance of total curvature (see the remark after Theorem 1), we have $0 < \mathbf{1}^tw_1 = \mathbf{1}^t\frac{\eta}{-z^t\xi}=0,$ a contradiction. Thus, $z^t\xi= 0$ and $D_1\eta = \mathbf{0}.$ Similarly, we have $D_2 \xi = \mathbf{0.}$ Therefore, $\eta \in \mynull D_1$ and $\xi \in \mynull D_2.$ 

\smallskip
We can thus write $\eta = c_1 \eta_1 + \cdots + c_{k_1} \eta_{k_1}$ and $\xi = d_1 \xi_1 + \cdots + d_{k_2}\xi_{k_2}$ where $c_i,d_j \in \bR$. This means that
$$\begin{bmatrix}
    \eta \\
    \xi
\end{bmatrix} = c_1 \begin{bmatrix}
    \eta_1 \\
    \mathbf{0}_{n_2}
\end{bmatrix} + \cdots + c_{k_1}\begin{bmatrix}
    \eta_{k_1}\\
    \mathbf{0}_{n_2}
\end{bmatrix} + d_1 \begin{bmatrix}
    \mathbf{0}_{n_1} \\
    \xi_1
\end{bmatrix} + \cdots + d_{k_2}\begin{bmatrix}
    \mathbf{0}_{n_1} \\
    \xi_{k_2}
\end{bmatrix}.$$
Thus, the vectors $$\begin{bmatrix}
        \eta_1 \\
        \mathbf{0}_{n_2}
    \end{bmatrix},...,\begin{bmatrix}
        \eta_{k_1} \\
        \mathbf{0}_{n_2}
    \end{bmatrix}, \begin{bmatrix}
        \mathbf{0}_{n_1} \\
        \xi_1
    \end{bmatrix},...,\begin{bmatrix}
        \mathbf{0}_{n_1} \\
        \xi_{k_2}
    \end{bmatrix}$$ form a basis of $\mynull D_G.$ This implies
    $$
    \dim \mynull D_G = \dim \mynull D_1 + \dim \mynull D_2,
    $$ as desired.

\end{proof}

\subsection{Nonexistence of Curvature} 
A necessary condition for the curvature to have desirable geometric properties is that the linear system $Dx = \mathbf{1}$ has a solution. Steinerberger raised the following problem.

\smallskip
\begin{quote}
\textbf{Problem (\cite{firsteig}).} It seems that for most graphs, the linear system $Dx=\mathbf{1}$ tends to have a solution. Why is that?
\end{quote}

He gave a sufficient condition for $Dx=\mathbf{1}$ to have a solution.
\begin{proposition}[\cite{firsteig}]
    Suppose $D \in \bR^{n \times n}_{\geq 0}$ has eigenvalues $\l_1 > 0 \geq \l_2 \geq \cdots \geq \l_n$ and eigenvector $D\eta = \l_1 \eta$ with $\|\eta\|_2 = 1.$ 
    If 
    \begin{equation}
    \label{eq:condition_inn}
        1 - \<\eta, \frac{1}{\sqrt{n}}\>^2 < \frac{|\l_2|}{\l_1 - \l_2},
        \end{equation}then the linear system $Dx = \mathbf{1}$ has a solution.
\end{proposition}
The proof is correct. However, this condition degenerates to the trivial condition of whether $D$ is invertible or not. Since if $\l_1 > 0 > \l_2 \geq \cdots \geq \l_n$, then $D$ is invertible. This implies that $Dx=\mathbf{1}$ has a solution. If $\l_1 > 0 = \l_2 \geq \cdots \geq \l_n$ then the right-hand side of inequality \eqref{eq:condition_inn} is $0$. The Cauchy--Schwarz inequality then implies that condition \eqref{eq:condition_inn} will never be satisfied.

\smallskip
By merging two graphs at a vertex, we can create graphs so that $Dx=\mathbf{1}$ does not have a solution.

\begin{theorem}
    Let $G_1$ and $G_2$ be two connected graphs so that $D_ix = \mathbf{1}$ does not have a solution. Let $H$ be obtained by adding an edge between $G_1$ and $G_2$, then performing an edge contraction on this edge. If $D_H$ is the distance matrix of $H$ then 
    $$D_Hx=\mathbf{1}$$ does not have a solution.
\end{theorem}
For other methods to construct graphs such that $Dx = \mathbf{1}$ has no solution, we refer the reader to \cite{dudarov2023image}.

\begin{proof}

As in the proof of Theorem 3, we assume $V(G_1) = \{u_1,...,u_{n_1}\}, V(G_2)=\{v_1,...,v_{n_2}\}$ and $\{u_{n_1},v_1\}$ is the edge added and contracted. The condition that $D_1x=\mathbf{1}$ has no solution is equivalent to $\mathbf{1}\not\in \myimage D_1 =(\mynull D_1)^\perp.$ This is equivalent to that there is $\eta \in \mynull D_1$ with $\<\eta, \mathbf{1}\> \neq 0.$ Similarly, we can find a vector $\xi \in \mynull D_2$ with $\<\xi,\mathbf{1}\> \neq 0.$ Our goal is to find a vector $\zeta \in \mynull D_H$ with $\<\zeta, \mathbf{1}\> \neq 0.$

Consider the vector
$$
\zeta = \a \begin{bmatrix}
    \eta \\
    \mathbf{0}_{n_2-1} 
\end{bmatrix}+ \begin{bmatrix}
    \mathbf{0}_{n_1} \\
    \Bar{\xi}
\end{bmatrix}+
(s + \xi_1)e_{n_1},
$$where $\Bar{\xi} = (\xi_2,...,\xi_{n_2})$, $e_{n_1}$ is the $n_1$-th coordinate vector in $\bR^{n_1+n_2-1}$, and $\a,s\in \bR$ are to be chosen. As in the proof of Theorem 3, let $y\in\bR^{n_1}$ be the last column of $D_1$ and $z\in\bR^{n_2-1}$ be the first column of $D_2$ without the first entry. Write
$$
D_H = \begin{bmatrix}
    D_1 & y\mathbf{1}^t + \mathbf{1}z^t \\
    \mathbf{1}y^t + z\mathbf{1}^t & \Bar{D_2}
\end{bmatrix}\in \bR^{(n_1+n_2-1)\times (n_1+n_2-1)},
$$where 
$$
D_2 = \begin{bmatrix}
    0 & z^t \\
    z & \Bar{D_2}
\end{bmatrix}.
$$
Then $D_2\xi = \mathbf{0}$ implies $z^t\Bar{\xi}=0$ and $\xi_1z + \Bar{D_2}\Bar{\xi} = \mathbf{0}_{n_2-1}.$ Therefore,

$$
D_H\zeta = \begin{bmatrix}
        \mathbf{0}_{n_1} \\
        \a \mathbf{1}y^t\eta + \a z\mathbf{1}^t\eta
    \end{bmatrix} + 
    \begin{bmatrix}
        y\mathbf{1}^t \Bar{\xi} \\
        -\xi_1z
    \end{bmatrix} + (s+\xi_1)\begin{bmatrix}
        y \\
        z
    \end{bmatrix}.
$$
Note that $D_1\eta = \mathbf{0}_{n_1}$ gives $y^t\eta = 0.$ Thus,
$$
D_H\zeta = \begin{bmatrix}
    (\mathbf{1}^t\xi + s)y \\
    (\mathbf{1}^t \eta \a + s)z
\end{bmatrix}.
$$
Set $s=-\mathbf{1}^t\xi$ and $\a = \frac{\mathbf{1}^t\xi}{\mathbf{1}^t\eta}.$ Note that $\a$ is well-defined since $\mathbf{1}^t \eta \neq 0.$ Then
$$
D_H\zeta = \mathbf{0}_{n_1+n_2-1}.
$$
In addition, we have
$$
\<\zeta, \mathbf{1}\> = \a \mathbf{1}^t\eta + \mathbf{1}^t\xi +s = \mathbf{1}^t\xi \neq 0.
$$
Therefore, 
$$
\mathbf{1}\not\in (\mynull D_H)^\perp = \myimage D_H.
$$ This implies that $D_H x= \mathbf{1}$ does not have a solution.
\end{proof}


\subsection{Perron Eigenvector of Distance Matrix}
In his work \cite{firsteig}, Steinerberger proves that if $\eta$ is the Perron eigenvector of the distance matrix of a graph (the first eigenvector whose entries are nonnegative), then $\<\eta,\mathbf{1}\>^2 \geq \frac{n}{2},$ where $n$ is the number of vertices. We provide a lower bound when the graph is a tree involving the number of leaves.

\begin{proposition}
    Let $T$ be a tree with $n$ vertices and $l$ leaves. Let $D$ be its distance matrix, $\l$ be its largest positive eigenvalue (Perron root), and $\eta$ be the Perron eigenvector of $D$ with $\|\eta\|_2 = 1.$ Then
    $$
    \<\eta,\mathbf{1}\>^2 > \frac{n}{2}(\frac{\l}{\l-l+1})+ \frac{n-l-1}{\l-l+2}.
    $$
\end{proposition}
\textbf{Example}. The star graph with $n$ vertices has $l=n-1$ leaves. The eigenvalue estimate of the Perron root (see for example, \cite[Theorem 8.1.22]{horn2012matrix} and \cite[Corollary 7]{ZHOU2007384}) gives
$$
\frac{2(n-1)^2}{n}=\frac{\sum_{i,j}D_{ij}}{n} \leq \l  \leq \max_i\sum_{j=1}^nD_{ij} = 2n - 3.
$$ Then the proposition above gives
$$
\<\eta,\mathbf{1}\>^2 > \frac{(n-1)^2}{n-1} = n-1.
$$

We follow the idea in the Theorem in \cite{firsteig} with some revision.
\begin{proof}[Proof of Proposition 4]
    Let $V=\{u_1,...,u_n\}$ be the vertices of the tree $T.$ Let $L \subset V$ be the leaves of $T.$ 
    Assume $u_k$ is not a leaf with $k$ fixed. Then
    \begin{align*}
        \l &= \sum_{i,j=1}^nd(u_i,u_j)\eta_i\eta_j \\
        &=\sum_{i\neq j}d(u_i,u_j)\eta_i\eta_j \\
        &\leq \sum_{i\neq j} (d(u_i,u_k)+d(u_k,u_j)) \eta_i\eta_j \\
        &= \sum_{i,j=1}^n(d(u_i,u_k) + d(u_k,u_j))\eta_i\eta_j - \sum_{i=1}^n2d(u_i,u_k)\eta_i^2.
    \end{align*}
    Therefore, $$
    \l + 2\sum_{i=1}^nd(u_i,u_k)\eta_i^2 \leq 2 \<\eta,\mathbf{1}\>\l \eta_k.
    $$
    Note that
    $$
    \sum_{i=1}^n d(u_i,u_k)\eta_i^2 = \sum_{i\neq k}d(u_i,u_k)\eta_i^2 \geq \sum_{i\neq k}\eta_i^2 = \|\eta\|_2^2 - \eta_k^2 = 1 - \eta_k^2.
    $$
    Thus, we get 
    $$
        \l+2 -2\eta_k^2 \leq 2\<\eta,\mathbf{1}\>\l\eta_k.
    $$
    Rearranging the terms and summing $k$ over all non-leaves, we get
    \begin{equation}
        \label{eq:nonleaf}
        \l(n-l) \leq 2 \<\eta,\mathbf{1}\>\l\sum_{k:u_k \not\in L}\eta_k + 2 \sum_{k:u_k\not\in L}\eta_k^2 - 2(n-l).
    \end{equation}
    
    On the other hand, suppose $u_k \in L$ is a leaf with $k$ fixed. If $i,j \neq k$ then
    $$
    d(u_i,u_j) \leq d(u_i,u_k) + d(u_k,u_j)- 2.
    $$
    To see this, assume that $u_k$ is adjacent to the vertex $u_{k'}.$ Then 
    \begin{align*}
    d(u_i,u_k)&=d(u_i,u_{k'})+1    \\
    d(u_j,u_k)&=d(u_j,u_k')+1.
    \end{align*}
    Thus,
    $$
    d(u_i,u_j) \leq d(u_i,u_{k'})+d(u_j,u_{k'}) = d(u_i,u_k)+d(u_j,u_k) - 2.
    $$
    Then we have
    \begin{align*}
        \l &= \sum_{i,j}\eta_i\eta_jd(u_i,u_j) \\
        &\leq \sum_{i,j \neq k}\eta_i\eta_j(d(u_i,u_k) + d(u_k,u_j) - 2) + 2 \sum_{i \neq k}\eta_i\eta_kd(u_i,u_k) + \eta_k^2d(u_k,u_k) \\
        &= 2(\<\eta,\mathbf{1}\> - \eta_k) \l \eta_k - 2(\<\eta,\mathbf{1}\> - \eta_k)^2 + 2\l \eta_k^2 \\
        &= (2\l + 4)\eta_k\<\eta,\mathbf{1}\> - 2 \<\eta,\mathbf{1}\>^2 -2\eta_k^2.
    \end{align*}
    By summing $k$ over all leaves, we get
    \begin{equation}
        \label{eq:leaf}
        \l l \leq (2\l + 4)\<\eta,\mathbf{1}\>\sum_{k:u_k \in L}\eta_k - 2\<\eta,\mathbf{1}\>^2l - 2\sum_{k:u_k\in L}\eta_k^2.
    \end{equation}
    Thus, adding equations \eqref{eq:nonleaf} and \eqref{eq:leaf}, we get
    $$
    \l n\leq (2\l - 2l)\<\eta, \mathbf{1}\>^2 + 4 \<\eta,\mathbf{1}\>\sum_{k:u_k \in L}\eta_k + 2 (\sum_{k:u_k \not\in L}\eta_k^2 - \sum_{k:u_k\in L}\eta_k^2) - 2(n-l).
    $$
    Since
    $$
        \sum_{k:u_k\not\in L}\eta_k^2 - \sum_{k:u_k \in L} \eta_k^2 < \sum_{k:u_k\not\in L}\eta_k^2 + \sum_{k:u_k \in L} \eta_k^2 = 1
        $$ and $$
        \sum_{k:u_k \in L}\eta_k < \<\eta,\mathbf{1}\>,
    $$
    we get 
    $$
    \l n < (2\l-2l+4)\<\eta,\mathbf{1}\>^2 + 2 - 2(n-l).
    $$
    Note that $l \leq n-1$ since $T$ is a tree. The eigenvalue estimate \cite[Theorem 8.1.22]{horn2012matrix} gives
    $$
    \l \geq \min_i \sum_{j=1}^nD_{ij} \geq n-1.
    $$
    Therefore, $\l-l+2 > 0.$ Thus,
    $$
    \<\eta,\mathbf{1}\>^2 > \frac{n}{2}(\frac{\l}{\l - l + 2}) + \frac{n-l-1}{\l-l+2}.
    $$
    
\end{proof}

\section*{Acknowledgement}
W.-C. Chen and M.-P. Tsui are supported by NSTC grant 109-2115-M-002-006.
The first author thanks Dr. Jephian C.-H. Lin for his comments on this paper and anonymous referees for their ideas on Proposition \ref{prop:curv_upper_half_total}.

\nocite{*}
\bibliographystyle{amsplain}
\bibliography{reference}

\end{document}